\documentclass[a4paper, 12pt]{article} 

\usepackage[utf8]{inputenc}
\usepackage[T1]{fontenc}
\usepackage[a4paper, margin=2.75cm]{geometry}
\usepackage{needspace, mathscinet}
\usepackage{blkarray} 
\usepackage{multirow}
\usepackage{libertine} 
\usepackage{inconsolata} 
\usepackage{nicefrac}
\usepackage{authblk, amsmath, amsthm, amssymb, graphicx, thmtools, thm-restate, mathtools}

\usepackage[colorlinks=true, citecolor=blue, linkcolor=blue, urlcolor=blue]{hyperref}
\usepackage{enumitem}
\usepackage[noabbrev,capitalize]{cleveref}
\crefname{claim}{Claim}{Claims}
\crefname{theorem}{Theorem}{Theorems}
\crefname{proposition}{Proposition}{Propositions}
\crefname{definition}{Definition}{Definitions}
\crefname{conjecture}{Conjecture}{Conjectures}
\crefname{corr}{Corollary}{Corollaries}

\newtheorem{example}{Example}

\declaretheorem[name=Theorem, numberwithin=section]{theorem}
\declaretheorem[name=Lemma, sibling=theorem]{lemma}

\declaretheorem[name=Definition, sibling=theorem]{definition}

\declaretheorem[name=Conjecture, sibling=theorem]{conjecture}

\declaretheorem[name=Observation, sibling=theorem]{observation}
\declaretheorem[name=Claim,
sibling=theorem]{claim}

\def\cqedsymbol{\ifmmode$\lrcorner$\else{\unskip\nobreak\hfil
\penalty50\hskip1em\null\nobreak\hfil$\lrcorner$
\parfillskip=0pt\finalhyphendemerits=0\endgraf}\fi}

\newenvironment{claimproof}[1][\proofname]
{%
	\proof[#1]%
}
{%
	\endproof%
}

\let\le\leqslant
\let\ge\geqslant
\let\leq\leqslant
\let\geq\geqslant

\let\OLDthebibliography\thebibliography
\renewcommand\thebibliography[1]{
  \OLDthebibliography{#1}
  \setlength{\parskip}{0pt}
  \setlength{\itemsep}{0pt plus 0.3ex}
}%
\AtBeginDocument{
   \def\MR#1{}
}%

\def\final{0}  
\def\iflong{\iffalse}
\ifnum\final=0  
\newcommand{\kristof}[1]{{\color{red}[\tiny{ \textbf{Kristóf:}  #1}]\marginpar{\color{red}*}}}
\newcommand{\benedek}[1]{{\color{blue}[\tiny{ \textbf{Benedek:}  #1}]\marginpar{\color{blue}*}}}
\newcommand{\kolja}[1]{{\color{orange}[\tiny{ \textbf{Kolja:}  #1}]\marginpar{\color{orange}*}}}
\newcommand{\luis}[1]{{\color{teal}[\tiny{ \textbf{Luis:}  #1}]\marginpar{\color{teal}*}}}
\else 
\newcommand{\kristof}[1]{}
\newcommand{\benedek}[1]{}
\newcommand{\kolja}[1]{}
\newcommand{\luis}[1]{}
\fi

\newcommand{\M}{\mathcal{M}}

\title{Embracing exchange sequences and oriented matroid polyhedron diameter}

\author{Kristóf Bérczi\thanks{Supported by the 
 Lend\"ulet Programme of the Hungarian Academy of Sciences (LP2021-1/2021), by the Ministry of Innovation and Technology of Hungary from the National Research, Development and Innovation Fund (ADVANCED 150556, 153096), and by the Dynasnet European Research Council Synergy project (ERC-2018-SYG 810115).}}
\affil{MTA-ELTE Matroid Optimization Research Group and HUN-REN-ELTE Egerváry Research Group, Department of Operations Research, ELTE Eötvös Loránd University, and HUN-REN Alfréd Rényi Institute of Mathematics, Budapest, Hungary.}
\author{Kolja Knauer\thanks{Supported through grant 
 PID2022-137283NB-C22 funded by MICIU/AEI/10.13039/501100011033 and by ERDF/EU and through the Severo Ochoa and María de Maeztu Program for Centers and Units of Excellence in R\&D (CEX2020-001084-M) and by the ANR project MIMETIQUE: ANR-25-CE48-4089-01.}}
\affil{Departament de Matemàtiques i Informàtica, Universitat de Barcelona, Centre de Recerca Matemàtica, Barcelona, Spain.}
\author{Luis Pedro Montejano\thanks{Supported by the Spanish Ministry of Science and
Innovation due to the grants PID2023-146925OB-I00 and PID2023-148202OB-C21.}}
\affil{Departament d'Enginyeria Inform\`atica i Matem\`atiques, Universitat Rovira i Virgili, Tarragona, Spain.}
\author{Benedek Nádor}
\affil{Department of Operations Research, ELTE Eötvös Loránd University}
\begin{document}

\maketitle

\begin{abstract}
    We reduce the embracing exchange distance of bases of oriented matroids to the metric of oriented matroid polyhedra. This allows us to disprove a recent conjecture of Caoduro, Khodamoradi, Paat, and Shepherd and its oriented-matroid extension introduced in~\cite{BN25}. On the other hand, we show that any two embracing bases of an oriented matroid of rank $r$ can be transformed into each other in at most $2r^{\log_2(r)+3}$ steps {and in at most $r$ steps in a graphic oriented matroid or a Lawrence oriented matroid, thus confirming the conjecture in these cases.}
\end{abstract}

\section{Introduction}
Caoduro, Khodamoradi, Paat, and Shepherd conjectured a $0$-embracing exchange bound for simplices; see~\cite[Conjecture~1.1]{BN25}. We use the following nondegenerate variant, cf.~\cite[Conjecture~4.1]{Nad26}:
\begin{conjecture}[{\cite[Conjecture~4.1]{Nad26}}]\label{conj:realizable}
Any two $d$-dimensional simplices that have the origin in their interior and such that no subset of size
at most $d$ vertices of their union has the origin in the relative interior of its convex hull\footnote{This has to be added to make the graph connected, since otherwise partitioning the corners
of a regular hexagon into two alternating triangles yields a disconnected instance.} can be
transformed into each other via at most $d+1$ vertex exchanges while always maintaining the origin
in the interior.
\end{conjecture}
B\'erczi and N\'ador~\cite{BN25} generalized the $0$-embracing exchange problem to the setting of
oriented matroids in the following way. Let $\mathcal{M}=(E,\mathcal{C})$ be an oriented matroid and
$e\in E$. Denote by $\mathcal{C}^+$ the set of positive circuits of $\mathcal{M}$ and by $\mathcal{C
}_e^+$ the set of \emph{$e$-positive circuits}, i.e., those $X\in\mathcal{C}^+$ with $X(e)=+$. For a
basis $B\in \mathcal{B}$ (of the underlying matroid) such that $e\notin B$, the \emph{fundamental
circuit} $C(B,e)$ is the unique $X\in \mathcal{C}$ whose support $\underline{C(B,e)}$ is contained in
$B\cup\{e\}$ and $C(B,e)_e=+$. We say that $B$ is \emph{$e$-embracing}\footnote{In~\cite{BN25} this
was defined differently but equivalent up to reorientation.} if $C(B,e)\in \mathcal{C}_e^+$.
The \emph{$e$-embracing distance} $d_e(B,B')$ of two $e$-embracing bases $B,B'$ is the minimum $k$ such
that there is a sequence $(B=B_0,B_1,\ldots,B_k=B')$ of $e$-embracing bases such that $|B_{i-1}\Delta
B_i|=2$ for all $i\in[k]$. B\'erczi and N\'ador~\cite{BN25} conjectured:
\begin{conjecture}[{\cite[Conjecture~1.2]{BN25}}]\label{conj:general}
The $e$-embracing exchange distance of two $e$-embracing bases of a rank $r$ oriented matroid is at most
$r$.
\end{conjecture}
In the realizable nondegenerate case, where all $e$-positive circuits have size $r+1$, \cref{conj:general}
specializes to \cref{conj:realizable}.

The set $\mathcal{C}_e^+$ of $e$-positive circuits constitutes the vertices of the $1$-skeleton $G(P_{\mathcal{M}^*,e})$ of the oriented matroid polyhedron $P_{\mathcal{M}^*,e}$ of the affine oriented matroid $(\mathcal{M}^*,e)$, see~\cite[Definition 10.1.1]{BLSWZ99}. Our main tool is that the oriented matroid polyhedron metric controls the $e$-embracing distance between
bases with distinct $e$-positive fundamental circuits up to a factor $r$; bases with the same
fundamental circuit have their ordinary basis-exchange distance:

\begin{restatable}{theorem}{main}
\label{thm:main}
   Let $\mathcal{M}=(E,\mathcal{C})$ be an oriented matroid of rank $r\geq 1$,  $e\in E$, and let $B,B'$ be $e$-embracing bases of ${\mathcal{M}}$. If $C(B,e)=C(B',e)$, then $d_e(B,B')=|B\setminus B'|$ and otherwise
   $$d_{G(P_{\mathcal{M}^*,e})}(C(B,e),C(B',e))\leq d_e(B,B')\leq rd_{G(P_{\mathcal{M}^*,e})}(C(B,e),C(B',e)).$$ Moreover, if all circuits in $\mathcal{C}^+_e$ are of size $r+1$, then the lower bound is attained with equality.
\end{restatable}

It might be that the lower bound in \cref{thm:main} indeed always is an equality - we believe that this is an interesting question. In any case, choosing the counterexamples to the polyhedral Hirsch conjecture of Klee and Walkup~\cite{KW67} and Todd~\cite{Tod80} as $P_{\mathcal{M}^*,e}$ we can disprove Conjectures~\ref{conj:realizable} and~\ref{conj:general}.

\begin{restatable}{corollary}{lb}
\label{cor:lb}
For every $r\geq 1$ there is a realizable oriented matroid
$\mathcal M$ of rank $r$ with $e$-embracing bases $B, B'$ such that $d_e(B,B')\geq \lfloor\frac{5}{4}r\rfloor$.
\end{restatable}

Note that a purely geometric disproof of \cref{conj:realizable} has also been provided by the last author in his Bachelor's thesis~\cite{Nad26}. 

Combining some oriented matroid restrictions with the fact that $G(P_{\mathcal{M}^*,e})$ obeys the proof of the upper bound on polyhedron diameter due to Kalai and Kleitman~\cite{KK92} we also get a positive result:

\begin{restatable}{corollary}{ub}
\label{cor:ub}
Let $\mathcal{M}=(E,\mathcal{C})$ be an oriented matroid of rank $r\geq 1$, $e\in E$, and let $B,B'$ be $e$-embracing bases of ${\mathcal{M}}$. Then $d_e(B,B')\leq 2r^{\log_2(r)+3}$.
\end{restatable}

As a first concrete positive example, we study graphic oriented matroids. Here a tree $T$ of a digraph $D$ is $e$-embracing if the fundamental cycle $C(T,e)$ is positively directed. In the graphic case, \cref{conj:general} holds, and the exchange sequence can moreover be chosen monotone with respect to the target basis:

\begin{restatable}{theorem}{graphic}
\label{thm:graphic}
Let $\M=(E,\mathcal{C})$ be a graphic oriented matroid of rank $r$, $e\in E$, and let $B, B'$ be $e$-embracing bases of $\M$. Then $d_e(B,B')\leq r$.
\end{restatable}

It would be natural to try and extend \cref{thm:graphic} to cographic or regular matroids and would be interested in seeing any progress into that direction..

Finally, we study Lawrence oriented matroids~\cite{lawrence1981unions,RS88}, i.e.,  uniform oriented matroids obtained as the union of uniform oriented matroids of rank $1$. These oriented matroids were used in \cite{R01} to provide upper bounds for McMullen's problem~\cite{L72}, in a generalization to $k$-neighborly polytopes \cite{GL15} and in relation to Roudneff's conjecture \cite{Rou91} in \cite{hernandez2025new,MR15}.
Here we prove \cref{conj:general} for these oriented matroids:

\begin{restatable}{theorem}{LOM}
\label{LOMsdiamr}
Let $\M=(E,\mathcal{C})$ be a Lawrence oriented matroid of rank $r$, $e\in E$, and let $B, B'$ be $e$-embracing bases of $\M$. Then $d_e(B,B')\leq r$.
\end{restatable}

\paragraph{Relation to previous versions.}

The present paper combines material from two earlier preprints. Bérczi and Nádor~\cite{BN25} introduced the oriented-matroid formulation of the embracing exchange problem and proved the graphic case. Knauer and Montejano~\cite{knauer2026embracing} developed the connection with oriented matroid polyhedra, obtained the counterexamples and the general upper bound, and proved the result for Lawrence oriented matroids. The present four-author version unifies these results in a common framework and includes several revisions and strengthenings.

\section{General bounds}

For basic notions on oriented matroids we refer to~\cite{BLSWZ99,and25,ziegler2024oriented}. 
Following~\cite[Definition 10.1.1]{BLSWZ99}, given an oriented matroid $\mathcal{M}=(E,\mathcal{L})$ with covectors $\mathcal{L}$ and $e\in E$ the \emph{oriented matroid polyhedron} is $P_{\mathcal{M},e}=\{X\in \mathcal{L}\mid \forall_{f\in E}:X(f)\geq 0  \text{ and }X_e=+\}$. The vertices of $P_{\mathcal{M},e}$ correspond to non-negative cocircuits $\mathcal{C}^{*+}_e$ of $\mathcal{M}$ that have positive $e$-coordinate. The graph $G(P_{\mathcal{M},e})$ is the $1$-skeleton of $P_{\mathcal{M},e}$. In particular, it has the same vertices and edges and can be viewed within the big face lattice of $\mathcal{M}$. The latter suffices to show that   $G(P_{\mathcal{M},e})$ is connected, see \cite[Lemma 3.2]{CK89}. 

For the sake of this paper we restate the adjacency in oriented matroid polyhedra in the dual setting.


\begin{definition}\label{lem:combi}
Let $\mathcal{M}=(E,\mathcal{C})$ be an oriented matroid and let $e\in E$. Then the vertex set of $G(P_{\mathcal{M}^*,e})$ is $\mathcal{C}_e^+$, where $X,X'\in\mathcal{C}_e^+$ are adjacent if and only if $Z
\in\mathcal{C}^+$ and $\underline{Z}\subseteq\underline{X}\cup\underline{X'}$ implies $Z\in\{X,X'\}$.
\end{definition}


\begin{lemma}\label{lem:lb}
Let $B$ and $B'$ be $e$-embracing bases of $\mathcal{M}$  with $|B\Delta B'|=2$. If $Z\in \mathcal{C}^+$ with $\underline{Z}\subseteq \underline{C(B,e)}\cup \underline{C(B',e)}$ then $Z=  C(B,e)$ or $Z=C(B',e)$. 
\end{lemma}
\begin{proof}
Let $B\setminus B'=\{f\}$. We start with an easy observation that we will use a couple of times.
\begin{claim}\label{claim:trivial}
If $Z\in \mathcal{C}$ with $\underline{Z}\subseteq \underline{C(B,e)}\cup \underline{C(B',e)}$ and $f\notin \underline{Z}$, then $Z=\pm C(B',e)$.
\end{claim}
\begin{claimproof}
    If $f\notin \underline{Z}$, then $\underline{Z}\subseteq B'\cup \{e\}$ concluding that $Z=\pm C(B',e)$ because $\underline{C(B',e)}$ is the unique circuit-support in $B'\cup \{e\}$.
\end{claimproof}

Now, let $Z\in \mathcal{C}^+$ with $\underline{Z}\subseteq \underline{C(B,e)}\cup \underline{C(B',e)}$. By \cref{claim:trivial} if $f\notin\underline{Z}$, then $\underline{Z}\subseteq B'\cup \{e\}$ concluding that $Z=C(B',e)$ and we are done. Again, by \cref{claim:trivial} if $f\notin\underline{C(B,e)}$, then $C(B,e)=C(B',e)$ and $Z=C(B,e)=C(B',e)$.

Hence, $f\in \underline{Z}\cap \underline{C(B,e)}$. Applying weak elimination to $C(B,e)$ and $-Z$ with respect to $f$, we obtain a circuit $\widetilde Z$
with $f\notin\underline{\widetilde Z}$ such that
$
\widetilde Z^+\subseteq \underline{C(B,e)}\setminus\{f\}$,
$\widetilde Z^-\subseteq \underline Z\setminus\{f\}$.
By \cref{claim:trivial}, $\widetilde Z=\pm C(B',e)$.

If $\widetilde Z=C(B',e)$, then $
\underline{C(B',e)}\subseteq\underline{C(B,e)}$.
Circuit incomparability gives $C(B',e)=C(B,e)$, and we are done.

If $\widetilde Z=-C(B',e)$, then
$\underline{C(B',e)}\subseteq\underline Z$. 
Again by circuit incomparability the two supports are equal, and since both $Z$ and $C(B',e)$ are
positive circuits, $Z=C(B',e)$. This concludes the proof.
\end{proof}

\begin{lemma}\label{lem:onlytwo}
  Let $X, Y\in \mathcal{C}^+$ be the only two positive circuits of $\mathcal{M}$ whose support is contained in $\underline{X}\cup\underline{Y}$. Then every mixed circuit $Z$ with $\underline{Z}\subseteq \underline{X}\cup\underline{Y}$ satisfies $\underline{Z}\supseteq \underline{X}\Delta\underline{Y}$.
\end{lemma}
\begin{proof}
    Let $Z$ be a mixed circuit of $\M$ with $\underline{Z}\subseteq \underline{X}\cup\underline{Y}$. 
    \begin{claim}\label{claim:inclusion}
        $\underline{X}\setminus \underline{Y}\subseteq Z^-$
   or $\underline{Y}\setminus \underline{X}\subseteq Z^-$.
   
    \end{claim}
    \begin{claimproof}
    For the given mixed circuit $Z$, choose a mixed circuit $W$ with $\underline W\subseteq\underline X\cup\underline Y$ and $W^-\subseteq Z^-$ such that $W^-$ is inclusion-minimal. It is enough to prove the
claim for $W$, since either $\underline X\setminus\underline Y\subseteq W^-$ or $\underline Y\setminus\underline X\subseteq W^-$ implies the corresponding inclusion in $Z^-$. Replacing $Z$ by $W
$, we may therefore assume that $Z^-$ is inclusion-minimal.

    First, show that $Z^-\cap \underline{X}\cap \underline{Y}=\emptyset$. Indeed, if $f\in Z^-\cap\underline X\cap\underline Y$, then weak
elimination of $Z$ and $X$ with respect to $f$ gives a circuit
$\widetilde Z$ with
$\widetilde Z^-\subseteq Z^-\setminus\{f\}$ and
$\underline{\widetilde Z}\subseteq\underline X\cup\underline Y$.
By the minimality of $Z^-$, the circuit $\widetilde Z$ is not mixed.
Replacing it by $-\widetilde Z$ if necessary, we obtain a positive
circuit supported in $\underline X\cup\underline Y$. Since
$f\notin\underline{\widetilde Z}$ but
$f\in\underline X\cap\underline Y$, it is different from $X$ and
$Y$, a contradiction.

    So suppose without loss of generality, that there is some $f\in Z^-\cap (\underline{X}\setminus\underline{Y})$. If there was also $g\in(X\setminus Y)\setminus Z^-$, then strong elimination applied to $Z$ and $X$ with respect to $f$, preserving $g$, yields a circuit $\widetilde Z$ such that
    \[
    g\in\widetilde Z^+,\qquad
    \widetilde Z^-\subseteq Z^-\setminus\{f\},\qquad
    \widetilde Z\subseteq X\cup Y.
    \]
    By the inclusion-minimality of $Z^-$, the circuit $\widetilde Z$ cannot be mixed. Since $g\in\widetilde Z^+$, it is therefore positive. Moreover, $\widetilde Z\neq X$ because $f\notin\widetilde Z$, while $\widetilde Z\neq Y$ because $g\in\widetilde Z$ and $g\notin Y$. This is a contradiction, yielding $\underline{X}\setminus \underline{Y}\subseteq Z^-$.
    \end{claimproof}
    
    Thus, by \cref{claim:inclusion} we have without loss of generality that $\underline{X}\setminus \underline{Y}\subseteq Z^-$. Then for the mixed circuit $-Z$ we have $\underline{X}\setminus \underline{Y}\subseteq (-Z)^+$ and therefore by \cref{claim:inclusion} we must have  $\underline{Y}\setminus \underline{X}\subseteq (-Z)^-=Z^+$. In particular, we have $\underline{Z}\supseteq \underline{X}\Delta\underline{Y}$.
\end{proof}

We observe the following for future use:
\begin{lemma}\label{obs:uniformcircuit}
Let $\mathcal{M}=(E,\mathcal{C})$ be an oriented matroid of rank $r$, let $e\in E$, and let
$X,X'\in\mathcal{C}_e^+$ with
$|\underline{X}|=|\underline{X'}|=r+1$. Then $X$ and $X'$ are adjacent in
$G(P_{\mathcal{M}^*,e})$ if and only if
$|\underline{X}\triangle\underline{X'}|=2$.
\end{lemma}

\begin{proof}
Suppose first that
$|\underline{X}\triangle\underline{X'}|=2$. Since $X$ and $X'$ are
circuits of size $r+1$, the sets
$B=\underline{X}\setminus\{e\}$ and
$B'=\underline{X'}\setminus\{e\}$ are bases of $\mathcal{M}$, with
$C(B,e)=X$ and $C(B',e)=X'$. Moreover, $|B\triangle B'|=2$.
By \cref{lem:lb}, $X$ and $X'$ are the only positive circuits whose
supports are contained in $\underline{X}\cup\underline{X'}$.
Hence they are adjacent by \cref{lem:combi}.

Conversely, suppose that $X$ and $X'$ are adjacent. Since they have
the same size, set
$t=|\underline{X}\setminus\underline{X'}|
 =|\underline{X'}\setminus\underline{X}|$.
As $X\neq X'$, we have $t\geq 1$. Suppose for a contradiction that
$t\geq2$, and choose
$x\in\underline{X}\setminus\underline{X'}$ and
$y\in\underline{X'}\setminus\underline{X}$. Set
$S=(\underline{X}\cup\underline{X'})\setminus\{x,y\}$.
Then
$|S|=r-1+t\geq r+1$, so $S$ is dependent and contains a circuit $D$.

The circuit $D$ cannot be positive or negative. Indeed, otherwise
$D$ or $-D$ would be a positive circuit supported in
$\underline{X}\cup\underline{X'}$ different from $X$ and $X'$,
contradicting their adjacency and \cref{lem:combi}. Thus $D$ is
mixed. Since adjacency also says that $X$ and $X'$ are the only
positive circuits supported in their union, \cref{lem:onlytwo}
implies
$\underline{X}\triangle\underline{X'}\subseteq\underline{D}$.
This is impossible, since $x,y\in
\underline{X}\triangle\underline{X'}$ but $x,y\notin\underline{D}$.
Therefore $t=1$, and hence
$|\underline{X}\triangle\underline{X'}|=2$.
 \end{proof}

We are ready to prove our main theorem:

\main*
\begin{proof}
If $C(B,e)=C(B',e)=X$, then both $B$ and $B'$ contain $\underline X\setminus\{e\}$. After contracting $\underline X\setminus\{e\}$ in the underlying matroid, ordinary basis exchange gives a path of length $|B\setminus B'|$ from $B$ to $B'$. Lifting this path, every intermediate basis still contains $\underline X\setminus\{e\}$, and hence has fundamental circuit $X$. Thus, $d_e(B,B')=|B\setminus B'|$.

If $C(B,e)\neq C(B',e)$, then the lower bound is a direct consequence of \cref{lem:combi} and \cref{lem:lb}.  

For the upper bound in this case, let
$C(B,e)=X_0,X_1,\ldots,X_k=C(B',e)$ be a shortest path in
$G(P_{\mathcal{M}^*,e})$. Thus,
$k=d_{G(P_{\mathcal{M}^*,e})}(C(B,e),C(B',e))$.

Consider an $e$-embracing basis $\widetilde B$ with $C(\widetilde B,e)=\widetilde X\in\mathcal C_e^+$, and
let $\widetilde Y\in\mathcal C_e^+$ be adjacent to $\widetilde X$ in $G(P_{\mathcal M^*,e})$. 

First choose $f \in \underline{\widetilde{Y}}\setminus\underline{\widetilde{X}}$ with $f\notin \widetilde{B}$. Such an element exists because otherwise $\underline{\widetilde{Y}}\setminus{e}\subseteq\widetilde{B}$ and $\underline{\widetilde{Y}}$ would be a second circuit in $\widetilde{B}\cup\{e\}$ contradicting $C(\widetilde{B},e)=\widetilde{X}\neq \widetilde{Y}$. Now set $S=(\underline{\widetilde X}\cup\underline{\widetilde Y})\setminus\{e,f\}$. If $S$ contained a circuit $D$, then $D$ could not be positive or negative, because otherwise $D$ or $-D$
would be a positive circuit supported in $\underline{\widetilde X}\cup\underline{\widetilde Y}$ different from the two adjacent vertices. Thus
$D$ would be mixed, and \cref{lem:onlytwo} would force $\underline{\widetilde X}\Delta\underline{\widetilde Y}\subseteq \underline{D}$, contradicting $f\notin S$. Hence $S$ is
independent and we can augment $S$ with elements of $\widetilde{B}$ yielding $\widetilde{B}'$.
Both $\widetilde B$ and $\widetilde B'$ contain $\underline{\widetilde X}\setminus\{e\}$, and $|\widetilde
B'\setminus\widetilde B|\leq|\underline{\widetilde Y}\setminus\underline{\widetilde X}|-1$. Hence,
basis exchange after contracting $\underline{\widetilde X}\setminus\{e\}$ gives an $e$-embracing
exchange sequence from $\widetilde B$ to $\widetilde B'$ of length at most $|\underline{\widetilde Y
}\setminus\underline{\widetilde X}|-1$; every basis on this path has fundamental circuit $\widetilde
X$.

Moreover, $\widetilde{B}'\cup\{f\}$ contains a unique circuit, which must be one of the mixed circuits on $\underline{\widetilde{X}}\cup\underline{\widetilde{Y}}\setminus\{e\}$. Indeed, weak elimination between $\widetilde X$ and $-\widetilde Y$ with respect to $e$ yields a circuit $D$ with
\[
\underline D\subseteq(\underline{\widetilde X}\cup\underline{\widetilde Y})\setminus\{e\}.
\]
The circuit $D$ is mixed: otherwise $D$ or $-D$ would be a positive circuit supported in $\underline{\widetilde X}\cup\underline{\widetilde Y}$ different from $\widetilde X$ and $\widetilde Y$,
contradicting their adjacency. Since $(\underline{\widetilde X}\cup\underline{\widetilde Y})\setminus
\{e\}\subseteq\widetilde B'\cup\{f\}$ and $\widetilde B'\cup\{f\}$ contains a unique circuit, $D$ is
that circuit. By \cref{lem:onlytwo}, $D$ contains $\underline{\widetilde X}\Delta\underline{\widetilde Y}$. Hence, for any $g\in\underline{\widetilde X}\setminus\underline{\widetilde Y}$, the
set $\widetilde B'\cup\{f\}\setminus\{g\}$ is a basis. It contains $\underline{\widetilde Y}\setminus
\{e\}$, so its fundamental circuit with $e$ is $\widetilde Y$.
Since $C(\widetilde{B}'\cup\{f\}\setminus\{g\}, e)=\widetilde{Y}$, this basis is at $e$-embracing distance at most $|\underline{\widetilde Y}\setminus\underline{\widetilde X}|$ from $\widetilde{B}$. 

Applying this reasoning along the sequence
$C(B,e)=X_0,X_1,\ldots,X_k=C(B',e)$, we obtain an $e$-embracing
basis $B''$ with $C(B'',e)=X_k$ and
$d_e(B,B'')\leq
\sum_{i=1}^k|\underline{X_i}\setminus\underline{X_{i-1}}|$.
Finally, $C(B'',e)=C(B',e)=X_k$, so the same contraction argument as
in the first paragraph gives
$d_e(B'',B')\leq |B''\setminus B'|
\leq r-|\underline{X_k}\setminus\{e\}|$.

For $i<k$ we have
$|\underline{X_i}\setminus\underline{X_{i-1}}|\leq r$, while
$|\underline{X_k}\setminus\underline{X_{k-1}}|
\leq|\underline{X_k}\setminus\{e\}|$. Hence
$d_e(B,B')\leq (k-1)r+
|\underline{X_k}\setminus\{e\}|+
r-|\underline{X_k}\setminus\{e\}|=rk$, proving the upper bound.

If all circuits in $\mathcal{C}_e^+$ have size $r+1$, then
\cref{obs:uniformcircuit} gives
$|\underline{X_i}\setminus\underline{X_{i-1}}|=1$ for every
$i\in[k]$, while $|\underline{X_k}\setminus\{e\}|=r$. Thus the construction above gives $d_e(B,B')\leq k$, while the
lower bound gives $k\leq d_e(B,B')$. Hence $d_e(B,B')=k$.
\end{proof}

\lb*
\begin{proof}
By the results of Klee and Walkup~\cite{KW67} and Todd~\cite{Tod80},
for every $k\geq1$ there is an unbounded simple $4$-dimensional
polyhedron $Q_k$ with $4+4k$ facets and diameter $5k$,
see~\cite{ES10}. Add another hyperplane $H_e$ to obtain a polytope
$Q_k^+$, and consider $Q_k^+$ as a tope of a realizable oriented
matroid $\mathcal{M}$ of rank $5$ on $5+4k$ elements, with all
hyperplanes oriented towards this tope. Set
$\mathcal{N}=\mathcal{M}^*$. Then the $e$-positive circuits of
$\mathcal{N}$ correspond to the vertices of $Q_k$, and
$G(P_{\mathcal{N}^*,e})=G(P_{\mathcal{M},e})$ is the graph of
$Q_k$.

Since $Q_k$ is simple, every $e$-positive circuit of $\mathcal{N}$
has size $4k+1=\operatorname{rk}(\mathcal{N})+1$. Choose
$X,X'\in\mathcal{C}_e^+(\mathcal{N})$ corresponding to vertices of
$Q_k$ at distance $5k$. Then
$B=\underline{X}\setminus\{e\}$ and
$B'=\underline{X'}\setminus\{e\}$ are $e$-embracing bases of
$\mathcal{N}$ with $C(B,e)=X$ and $C(B',e)=X'$. By the equality
case of \cref{thm:main},
$d_e(B,B')=d_{G(P_{\mathcal{N}^*,e})}(X,X')=5k$.
Since $\operatorname{rk}(\mathcal{N})=4k$, this proves the result
when the rank is divisible by $4$.

Now let $r=4k+s$, where $s\in\{0,1,2,3\}$ and $k\geq1$. Take the
direct sum of $\mathcal{N}$ with $s$ oriented rank-one matroids whose
underlying matroids are $U_{1,2}$, and in each added summand choose
the two distinct singleton bases. The resulting bases are
$e$-embracing. Moreover, after deleting repetitions, the projection of any $e$-embracing
exchange sequence onto the first summand gives an $e$-embracing
exchange sequence from $B$ to $B'$, while each of the $s$ added
summands must be changed at least once. Hence their $e$-embracing
distance is at least
$5k+s=\lfloor 5r/4\rfloor$.

For $r\leq3$, take a positive loop $e$ together with $r$ copies of
$U_{1,2}$ and choose the two opposite bases in these rank-one
summands. Every basis is then $e$-embracing and the two chosen bases
have exchange distance $r=\lfloor5r/4\rfloor$.
\end{proof}


\begin{lemma}\label{lem:restriction-distance}
Let $\mathcal{M}=(E,\mathcal{L})$ be an oriented matroid, let $e\in E$,
and put $P=P_{\mathcal{M},e}$. Let $X$ be a vertex of $P$, let
$e\in F\subseteq E$, and let $k\geq0$. Suppose that every vertex $Z$
of $P$ with $d_{G(P)}(X,Z)\leq k$ satisfies $Z(f)\neq0$ for every
$f\in E\setminus F$. Then every such $Z$ restricts to a vertex of
$P_{\mathcal{M}|F,e}$ and
$d_{G(P_{\mathcal{M}|F,e})}(X|_F,Z|_F)\geq d_{G(P)}(X,Z)$.
\end{lemma}

\begin{proof}
Put $\mathcal{N}=\mathcal{M}^*$ and $D=E\setminus F$. By
\cref{lem:combi}, we may identify the vertices of $P$ with the
$e$-positive circuits of $\mathcal{N}$, and the vertices of
$P_{\mathcal{M}|F,e}$ with the $e$-positive circuits of
$\mathcal{N}/D=(\mathcal{M}|F)^*$.

If $U$ is a vertex of $P$ with no zero coordinate in $D$, then by the circuit description of contraction 
$D\subseteq\underline U$ and $U|_F$ is a circuit of $\mathcal{N}/D$.
Indeed, otherwise a circuit of $\mathcal{N}/D$ properly contained in
$\underline U\setminus D$ could be lifted to a circuit of
$\mathcal{N}$ properly contained in $\underline U$, contradicting
circuit incomparability. Thus $U|_F$ is a vertex of
$P_{\mathcal{M}|F,e}$.


We claim the following local lifting property. Let $U$ be a vertex of $P$ with $D\subseteq U$, and let $V$ be adjacent to $U|F$ in $G(P_{M|F,e})$. Then either there is a neighbor $\widetilde V$ of $U$ in $G(P)$ with $D\subseteq\widetilde V$ and $\widetilde V|F=V$, or $U$ has a neighbor $\widetilde V$ with $\widetilde V(d)=0$ for some $d\in D$.

To prove the claim, choose a circuit $V_0$ of $\mathcal{N}$ with
$V_0|_F=\overline V$ and orient it so that $(V_0)_e=+$. Thus
$V_0^-\subseteq D$. If $V_0$ is positive, put $V_1=V_0$. Otherwise,
choose a mixed circuit $W$ with $e\in W^+$,
$\underline W\subseteq\underline U\cup\underline{V_0}$ and
$W^-\subseteq V_0^-$ such that $W^-$ is inclusion-minimal. For
$d\in W^-$, strong elimination between $U$ and $W$ with respect to
$d$, preserving $e$, yields a circuit $V_1$ with $(V_1)_e=+$,
$d\notin\underline{V_1}$,
$\underline{V_1}\subseteq\underline U\cup\underline{V_0}$ and
$V_1^-\subseteq W^-\setminus\{d\}$. By minimality of $W^-$, the
circuit $V_1$ is positive.

Set $S=\underline U\cup\underline{V_0}$. The vertices $U$ and $V_1$
belong to $G(P_{(\mathcal{N}|S)^*,e})$, which is connected by
\cite[Lemma~3.2]{CK89}. Let $U_1$ be the first vertex after $U$ on a
path from $U$ to $V_1$. The vertices $U$ and $U_1$ are also adjacent
in $G(P)$: indeed, every positive circuit of $\mathcal{N}$ supported
in $\underline U\cup\underline{U_1}$ is contained in $S$, and hence
\cref{lem:combi} applied to $\mathcal{N}|S$ gives the assertion.

If $U_1(d)=0$ for some $d\in D$, the second alternative of the claim
holds. Otherwise $D\subseteq\underline{U_1}$, so $U_1|_F$ is an
$e$-positive circuit of $\mathcal{N}/D$ whose support is contained in
$\underline{U|_F}\cup\underline{\overline V}$. Since $U|_F$ and
$\overline V$ are adjacent, \cref{lem:combi} gives
$U_1|_F\in\{U|_F,\overline V\}$. The first possibility would imply
$U_1=U$, since both positive circuits contain all of $D$. Hence
$U_1|_F=\overline V$, proving the claim.

Now let $Z$ be a vertex with $\ell=d_{G(P)}(X,Z)\leq k$, and suppose
for a contradiction that there is a path of length $m<\ell$ from
$X|_F$ to $Z|_F$ in $G(P_{\mathcal{M}|F,e})$. Starting with $X$,
lift its edges successively using the claim. If the second alternative
occurs after $j$ lifted edges, then $X$ has a path of length at most
$j+1\leq m<\ell\leq k$ to a vertex having a zero coordinate in $D$,
contrary to the hypothesis. Thus every edge lifts, and we obtain a
path of length $m$ in $G(P)$ from $X$ to a vertex $Z'$ with
$D\subseteq\underline{Z'}$ and $Z'|_F=Z|_F$.

The hypothesis also gives $D\subseteq\underline Z$. Since $Z$ and
$Z'$ are positive circuits, they agree on $F$ and are positive on
every element of $D$, so $Z'=Z$. Hence $d_{G(P)}(X,Z)\leq m<\ell$,
a contradiction.
\end{proof}

The following is an extension of the proof of~\cite{KK92} to oriented matroids.

\begin{lemma}\label{lem:ub}
Let $\mathcal{M}=(E,\mathcal{L})$ be an oriented matroid of rank
$r\geq2$, let $|E|=n\leq2^{r-1}$ and $e\in E$, and let
$P_{\mathcal{M},e}$ be an oriented matroid polyhedron. For any two
vertices $X,X'$ of $G(P_{\mathcal{M},e})$ we have
$d_{G(P_{\mathcal{M},e})}(X,X')
\leq(n-1)^{\log_2(r-1)+1}$.
\end{lemma}

\begin{proof}
We may delete loops different from $e$, since this does not change
$G(P_{\mathcal{M},e})$. Moreover, $e$ is not a loop because
$P_{\mathcal{M},e}$ has vertices. Put $P=P_{\mathcal{M},e}$ and
$N=|E\setminus\{e\}|\leq n-1$, and denote by $H(N,r)$ the maximum
diameter of an oriented matroid polyhedron of rank at most $r$ with
at most $N$ elements different from its distinguished element.

Fix vertices $X,X'$ of $P$ and put $\ell=d_{G(P)}(X,X')$. For
$Y\in\{X,X'\}$ and $0\leq j\leq \ell$, let $E_j(Y)$ be the set of
$f\in E\setminus\{e\}$ for which there is a vertex $Z$ of $P$ with
$d_{G(P)}(Y,Z)\leq j$ and $Z(f)=0$. If $|E_0(Y)|>N/2$, set
$k_Y=0$; otherwise let $k_Y$ be maximal such that
$0\leq k_Y\leq \ell$ and
$|E_{k_Y}(Y)|\leq\lfloor N/2\rfloor$.

We claim that $k_Y\leq H(\lfloor N/2\rfloor,r)$. This is immediate
if $k_Y=0$. Otherwise choose a vertex $W$ with
$d_{G(P)}(Y,W)=k_Y$ and put $F_Y=E_{k_Y}(Y)\cup\{e\}$. By the
definition of $E_{k_Y}(Y)$, every vertex in the $k_Y$-ball around
$Y$ has no zero coordinate outside $F_Y$. Hence
\cref{lem:restriction-distance} gives
\[
k_Y=d_{G(P)}(Y,W)
\leq d_{G(P_{\mathcal{M}|F_Y,e})}(Y|_{F_Y},W|_{F_Y})
\leq H(\lfloor N/2\rfloor,r),
\]
proving the claim.

If $k_X=\ell$ or $k_{X'}=\ell$, the claim already gives
$\ell\leq H(\lfloor N/2\rfloor,r)$. We may therefore assume that
$k_X,k_{X'}<\ell$. By maximality, both $E_{k_X+1}(X)$ and
$E_{k_{X'}+1}(X')$ have more than $N/2$ elements, and hence they
contain a common element $g$. Thus there are vertices
$\widetilde X,\widetilde X'$ of $P$ with
$\widetilde X(g)=\widetilde X'(g)=0$,
$d_{G(P)}(X,\widetilde X)\leq k_X+1$ and
$d_{G(P)}(\widetilde X',X')\leq k_{X'}+1$.

The vertices with $g$-coordinate zero form a face of
$P_{\mathcal M,e}$ naturally identified with
$P_{\mathcal M/g,e}$. In particular, its graph is the corresponding
face subgraph of $G(P)$. Since $g$ is not a loop,
$\operatorname{rk}(\mathcal{M}/g)=r-1$, and this polyhedron has at
most $N-1$ elements different from $e$. Consequently,
\[
\ell\leq k_X+k_{X'}+2+H(N-1,r-1),
\]
and therefore
\[
H(N,r)\leq
2H(\lfloor N/2\rfloor,r)+H(N,r-1)+2.
\]
With $d=r-1$, this is the standard Kalai--Kleitman recurrence. Since
$N\leq n-1<2^{r-1}$, the calculation in
\cite[Theorem~2.6]{KS09bis} yields
$H(N,r)\leq N^{\log_2(r-1)+1}
\leq(n-1)^{\log_2(r-1)+1}$.
\end{proof}

\ub*
\begin{proof}
First note that we may assume without loss of generality that $E=B\cup B'\cup\{e\}$, since other elements can be removed from $\mathcal{M}$ and the $e$-distance can only increase. In particular, $n\coloneqq|E|\le 2r+1$.

For $r\le 2$, there are at most $\binom{2r}{r}$ bases avoiding $e$, and Theorem~\ref{thm:main} implies that the graph of $e$-embracing bases is connected. Hence
$
d_e(B,B')\le \binom{2r}{r}-1\le 2r^{\log_2(r)+3},$
 and we may assume that $r\ge 3$.

By adding loops if necessary, we may furthermore assume that $n=2r+1$. This does not change the bases, the fundamental circuits $C(B,e)$ and $C(B',e)$, or the $e$-embracing exchange graph. Let $X=C(B,e)$ and $X'=C(B',e)$. Then $r^*=\operatorname{rk}(\mathcal{M}^*)=n-r=r+1$, and, since $r\ge 3$, $n=2r+1\le 2^r=2^{r^*-1}$. 
If $X=X'$, then \cref{thm:main} gives
$d_e(B,B')=|B\setminus B'|\le r$, and we are done. Hence we may
assume that $X\neq X'$.
Thus, by Lemma~\ref{lem:ub}, $d_{G(P_{\mathcal{M}^*,e})}(X,X')\le (n-1)^{\log_2(r^*-1)+1}$. 
Since $n=2r+1$ and $r^*=r+1$, we get $d_{G(P_{\mathcal{M}^*,e})}(X,X')\le (2r)^{\log_2(r)+1}$. 
Combining this with Theorem~\ref{thm:main} we get $d_e(B,B')\le r(2r)^{\log_2(r)+1}=2r^{\log_2(r)+3}$, 
concluding the proof.
\end{proof}

\section{Graphic oriented matroids}
\label{sec:graphic}

We next consider the graphic case. Let $D=(V,E)$ be a directed graph, and let $T$ be a spanning tree of its underlying undirected graph. For $u,v\in V$, we denote by $T[u,v]$ the unique path in $T$ from $u$ to $v$, traversed from $u$ to $v$, and use the same notation for its arc set. Thus, the arcs of $T[u,v]$ need not all be oriented in the direction of the traversal. We call $T$ \emph{$uv$-embracing} if $T[u,v]$ is a directed path from $u$ to $v$.

The directed graph $D$ induces an oriented graphic matroid on $E$. A choice of traversal of a cycle determines a signed circuit by assigning sign $+$ to the arcs whose orientation agrees with the traversal and sign $-$ to the remaining arcs. Let $e=ts\in E$, and let $T$ be a spanning tree with $e\notin T$. Under the convention from the introduction that $C(T,e)_e=+$, the tree $T$ is $e$-embracing if and only if it is $st$-embracing. Indeed, the corresponding traversal of the fundamental cycle first uses $e$ from $t$ to $s$ and then follows $T[s,t]$ from $s$ to $t$.

The oriented constraint may force an exchange sequence to be longer than an ordinary shortest basis-exchange sequence, even in the graphic case.

\begin{example}\label{ex:graphic-not-shortest}
Let $D=(V,E)$ be the directed graph with
\begin{equation*}
V=\{s,t,u,v,w\}
\qquad\text{and}\qquad
E=\{su,sv,uw,vw,ut,vt\}.
\end{equation*}
Let
\begin{equation*}
A=\{su,uw,vw,ut\}
\qquad\text{and}\qquad
B=\{sv,uw,vw,vt\}.
\end{equation*}
Both $A$ and $B$ are $st$-embracing spanning trees. Since $|A\setminus B|=2$, an exchange sequence of length two would have to replace, at each step, an arc of $A\setminus B$ by an arc of $B\setminus A$. The first deleted arc would therefore be either $su$ or $ut$. However, neither arc can be replaced by $sv$ or $vt$ so as to obtain an $st$-embracing spanning tree. Hence, no $st$-embracing exchange sequence of length two transforms $A$ into $B$.

On the other hand, the sequence
\begin{equation*}
A,\quad A-vw+sv,\quad A-vw+sv-ut+vt,\quad B
\end{equation*}
is an $st$-embracing exchange sequence of length three, where the last exchange replaces $su$ by $vw$.
\end{example}

We call an $st$-embracing exchange sequence $T_0,T_1,\ldots,T_q=B$ \emph{monotone with respect to $B$} if $|T_i\cap B|\geq |T_{i-1}\cap B|$ for every $i\in[q]$. The following result proves that the rank bound holds for graphic oriented matroids. In fact, the exchange sequence can be chosen monotone with respect to the target tree.

\begin{theorem}
\label{thm:graphic-monotone}
Let $D=(V,E)$ be a directed graph, let $s,t\in V$ be distinct vertices, and let $A$ and $B$ be $st$-embracing spanning trees of $D$. Then there exist $st$-embracing spanning trees
\begin{equation*}
A=T_0,T_1,\ldots,T_q=B
\end{equation*}
with $q\leq |V|-1$ such that $|T_{i-1}\mathbin{\Delta}T_i|=2$ and $|T_i\cap B|\geq |T_{i-1}\cap B|$ for every $i\in[q]$.
\end{theorem}

\begin{proof}
Write the directed path in $B$ from $s$ to $t$ as
\begin{equation*}
B[s,t]=(f_1,\ldots,f_k),
\end{equation*}
where $f_i=v_{i-1}v_i$ for $i\in[k]$, with $v_0=s$ and $v_k=t$. We first transform $A$ into an $st$-embracing spanning tree containing $B[s,t]$ by inserting the arcs $f_1,\ldots,f_k$ in this order. The following claim gives one step of this procedure.

\begin{claim}
\label{claim:graphic-step}
Let $T$ be an $st$-embracing spanning tree such that $B[s,v_{i-1}]\subseteq T$, and suppose that $f_i\notin T$. Then there exists $g\in T$ such that $T'=T-g+f_i$ is an $st$-embracing spanning tree and $B[s,v_i]\subseteq T'[s,v_i]$.
\end{claim}

\begin{claimproof}
Adding $f_i$ to $T$ creates the cycle $T[v_{i-1},v_i]+f_i$. Traverse $T[v_{i-1},v_i]$ from $v_{i-1}$ to $v_i$, and let $g$ be the last arc of this path that does not belong to $B[v_i,t]$. Such an arc exists because the first arc of $T[v_{i-1},v_i]$ is incident with $v_{i-1}$, whereas $v_{i-1}$ does not lie on $B[v_i,t]$.

Every arc following $g$ on $T[v_{i-1},v_i]$ belongs to $B[v_i,t]$. Consequently, these arcs form $B[v_i,v_j]$, traversed from $v_j$ to $v_i$, for some $j\in[i,k]$. In particular, $g$ has endpoint $v_j$. Moreover, $g\notin B[s,v_{i-1}]$, since every arc of $B[s,v_{i-1}]$ has both endpoints among $v_0,\ldots,v_{i-1}$. Hence
\begin{equation*}
T'=T-g+f_i
\end{equation*}
is a spanning tree containing $B[s,v_j]$.

If $g\notin T[s,t]$, then the directed path $T[s,t]$ remains in $T'$, and hence $T'$ is $st$-embracing. Suppose that $g\in T[s,t]$. Since $B[s,v_{i-1}]$ remains in $T-g$, the vertices $s$ and $v_{i-1}$ lie in the same component of $T-g$, whereas $v_j$ and $v_i$ lie in the other component. As $g$ also belongs to the directed path $T[s,t]$, it is traversed toward $v_j$ and therefore enters $v_j$. Consequently, the concatenation of the directed paths $B[s,v_j]$ and $T[v_j,t]$ is a directed path from $s$ to $t$ in $T'$. Thus, $T'$ is $st$-embracing.

Finally, $B[s,v_i]\subseteq T'$. Since $T'$ is a tree, this implies $T'[s,v_i]=B[s,v_i]$.
\end{claimproof}

Repeatedly applying \cref{claim:graphic-step} yields, after at most $|B[s,t]|$ exchanges, an $st$-embracing spanning tree $A'$ with $B[s,t]\subseteq A'$. Since $A'$ is a tree, this implies $A'[s,t]=B[s,t]$.

We now transform $A'$ into $B$ without changing the directed path $B[s,t]$. Starting with $T=A'$, while $T\neq B$, choose any $g\in T\setminus B$. By basis exchange in the underlying graphic matroid, there exists $f\in B\setminus T$ such that $T-g+f$ is a spanning tree. Since $B[s,t]\subseteq T$ and $g\notin B$, the deleted arc does not belong to $B[s,t]$. Therefore, $T-g+f$ still contains the directed path $B[s,t]$ and is $st$-embracing. Each such exchange increases $|T\cap B|$ by one, so after exactly $|B\setminus A'|$ exchanges we reach $B$.

Every exchange in the first phase inserts an arc of $B$, so the size of the intersection with $B$ never decreases. The total number of exchanges is at most
\begin{equation*}
|B[s,t]|+|B\setminus A'|\leq |B|=|V|-1,
\end{equation*}
which proves the theorem.
\end{proof}

To deduce Theorem~\ref{thm:graphic}, let $D=(V,E)$ be a directed graph representing the given graphic oriented matroid. By identifying one vertex from each connected component of $D$, we may assume that $D$ is connected without changing its oriented graphic matroid. Hence, the bases are precisely the spanning trees of the underlying graph.

If $e$ is a loop, every basis is $e$-embracing and the result follows from ordinary basis exchange. Otherwise write $e=ts$ with $s\neq t$. The $e$-embracing bases of the oriented graphic matroid are precisely the $st$-embracing spanning trees. Since $D$ is connected, its graphic matroid has rank $|V|-1$. Hence, Theorem~3.1 gives $d_e(A,B)\leq r$ in the graphic case. If $A\cap B=\emptyset$, then every exchange sequence transforming $A$ into $B$ has length at least $|A|=|V|-1$. Thus, in this case, Theorem~\ref{thm:graphic-monotone} yields a shortest possible $st$-embracing exchange sequence.


\graphic*

\section{Lawrence oriented matroids}

{The aim of this section is to prove that the $e$-embracing distance is at most $r$ for Lawrence oriented matroids (LOMs). 
The proofs of the preliminary results rely on the fact that $[n]$ is a totally ordered set with its natural increasing order. Nevertheless, considering the reverse order yields analogous statements that will also be used in the proof of \Cref{LOMsdiamr} (specifically, assertions (b) of each lemma). Since these proofs follow from the same arguments, we omit the details.} For the whole section, if $X$ is a circuit, we write
$\underline{X}=\{x_1<\cdots<x_{r+1}\}\subseteq[n]$, and similarly
for $Y$.

For our purposes, the following definition of LOMs is more convenient than the one in \cite{lawrence1981unions} (see also \cite[Section 7.6]{BLSWZ99} and \cite{R01,RS88}). 
Given a matrix $A{=(a_{i,j})}\in\{+1,-1\}^{r\times n}$ the \emph{Lawrence oriented matroid} (LOM) $\M_A$ is uniform of rank $r$ and the signs of the elements of any circuit $X$ of $\M_A$, can be obtained by the equation $X_{x_{i}}=X_{x_{i+1}} \Longleftrightarrow a_{i,x_i}=-a_{i,x_{i+1}}$, see \cite[Lemma 2.1]{R01}.  
To simplify the notation, we denote $\mathcal{C}^+:=\mathcal{C}^+(A)$, $\mathcal{C}_e^+:=\mathcal{C}_e^+(A)$ and $G(P_{\mathcal{M}_A^*,e}):=G_e(A)$. {We begin with the following remark:}


\begin{observation}\label{positivecircuit}
Let $\M_A$ be a LOM of rank $r$. A set
$\{x_1<\cdots<x_{r+1}\}\subseteq[n]$ is the support of a positive
circuit if and only if
$a_{i,x_i}=-a_{i,x_{i+1}}$ for every $i=1,\ldots,r$.
\end{observation}


Given $X\in \mathcal{C}^+(A)$ and $1\le y\le x_j$, $1\le j\le r$,
we define the integer $x_j(y)\in\{x_j,x_{j+1}\}$ such that
$a_{j,y}=-a_{j,x_j(y)}$. Note that $x_j(y)$ exists since
$a_{j,x_j}=-a_{j,x_{j+1}}$. For $s>j$, define
$x_j^s(y):=x_s(x_j^{s-1}(y))$ and $x_j^j(y):=x_j(y)$. 
As an example,  \Cref{thmLOM} shows a matrix $A$ and $X\in \mathcal{C}^+(A)$ where
$\underline{X}=\{4,5,6,7,9,11,13,14\}$. Thus, if  $y=1$  we obtain $x_1(y)=4$, $x_1^2(y)=6$, $x_1^3(y)=7$, $x_1^4(y)=9$, $x_1^5(y)=11$, $x_1^6(y)=13$ and $x_1^7(y)=14$.
The following lemma will be useful in this section.

\begin{lemma}\label{intercambio}
Let $\M_A$ be a LOM of rank $r$, let $X\in\mathcal{C}^+(A)$, and
let $y\in[n]\setminus\underline{X}$. Then there exist
$Z\in\mathcal{C}^+(A)$ and $x\in\underline{X}$ such that
$\underline{Z}=(\underline{X}\cup\{y\})\setminus\{x\}$, satisfying the following conditions.
\begin{enumerate}[label=(\alph*)]
  \item  If $x_i<y< x_{i+1}$ and $a_{i,x_i}=- a_{i,y}$, $1\le i\le r$, then  $x\in \{x_{i+1},\ldots,x_{r+1}\}$. If $1\le y<x_{1}$ then    
 $x\in \underline{X}$. \label{it:a}
  \item If $x_i<y< x_{i+1}$ and $a_{i,y}=-a_{i,x_{i+1}}$, $1\le i\le r$, then $x\in \{x_{1},\ldots,x_{i}\}$. 
  If $n\ge y>x_{r+1}$, then $x\in\underline{X}$. \label{it:b}
\end{enumerate}
\end{lemma}
\begin{proof}
\ref{it:a} If $i=r$, then
$\underline{Z}=\{x_1,\ldots,x_r,y\}$ satisfies the conditions of the
lemma by \cref{positivecircuit}. We may therefore assume that
$i\le r-1$.

We show that
\[
\underline{Z}
=\{x_1,\ldots,x_i,y,
x_{i+1}^{i+1}(y),x_{i+1}^{i+2}(y),\ldots,x_{i+1}^{r}(y)\}
\]
satisfies the conditions of the lemma. Indeed, by
\cref{positivecircuit},
$a_{j,x_j}=-a_{j,x_{j+1}}$ for every $1\le j\le i-1$, and
$a_{i,x_i}=-a_{i,y}$ by hypothesis. Furthermore,
$a_{i+1,y}=-a_{i+1,x_{i+1}^{i+1}(y)}$ and
\[
a_{s,x_{i+1}^{s-1}(y)}
=-a_{s,x_{i+1}^{s}(y)}
\]
for every $s\ge i+2$. Hence, \cref{positivecircuit} yields
$Z\in\mathcal{C}^+(A)$, and
$\underline{Z}=\underline{X}\cup y-x$ for some
$x\in\{x_{i+1},\ldots,x_{r+1}\}$.

Now suppose that $1\le y<x_1$, and consider
\[
\underline{Z}
=\{y,x_1^1(y),x_1^2(y),\ldots,x_1^r(y)\}.
\]
By definition,
$a_{1,y}=-a_{1,x_1^1(y)}$ and
$a_{i,x_1^{i-1}(y)}=-a_{i,x_1^i(y)}$ for every
$i=2,\ldots,r$. Thus $Z\in\mathcal{C}^+(A)$ by
\cref{positivecircuit}. Since
$\{x_1^1(y),\ldots,x_1^r(y)\}\subseteq\underline{X}$, we have
$\underline{Z}=\underline{X}\cup y-x$ for some
$x\in\underline{X}$, concluding the proof of \ref{it:a}.

Assertion \ref{it:b} follows analogously by reversing the order.
\end{proof}

Let $\underline{X}=\{x_1,\ldots,x_{8}\}=\{4,5,6,7,9,11,13,14\}$ in matrix $A$ of \Cref{thmLOM}. Then $X\in \mathcal{C}^+(A)$, and if $y=1$, it follows from the previous lemma that $Z\in \mathcal{C}^+(A)$, where  $\underline{Z}=\underline{X}\cup y -x_2=\{y,x_1(y),x_1^2(y),\ldots,x_1^7(y)\}=\{1,4,6,7,9,11,13,14\}$.

{We write $X\sim_{G_e(A)} Y$ to indicate that $X$ and $Y$ are adjacent in  $G_e(A)$.  We will frequently make use of 
 the fact that $X\sim_{G_e(A)} Y$   if  $|\underline{X}\Delta \underline{Y}|=2$ since $\M_A$ is a uniform oriented matroid (\cref{obs:uniformcircuit}).}

\begin{lemma}\label{e_primero_ultimo}
Let  $\M_A$ be a LOM of rank $r$   and let $X,Y\in \mathcal{C}_e^+(A)$ such that $e=x_{i}=y_{i}$, $1\le i\le r+1$. Then the following hold.
\begin{enumerate}[label=(\alph*)]
    \item If $j\coloneqq\min\{\ell\in[r+1]\mid x_\ell\neq y_\ell\}$ satisfies $i<j$, then $d_{G_e(A)}(X,Y)\le r+2-j$. \label{it:aa}
    \item If $j\coloneqq\max\{\ell\in[r+1]\mid x_\ell\neq y_\ell\}$ satisfies $j<i$, then $d_{G_e(A)}(X,Y)\le j$. \label{it:bb}
\end{enumerate}

\end{lemma}
\begin{proof}
\ref{it:aa} We prove the assertion by descending induction on $j$. If
$j=r+1$, then $X$ and $Y$ differ in exactly one element, and hence
$d_{G_e(A)}(X,Y)\le1$ by \cref{obs:uniformcircuit}.

Suppose now that $j\le r$ and that the assertion holds whenever the
first differing index is greater than $j$. By interchanging $X$ and
$Y$ if necessary, suppose that $y_j<x_j$. By the minimality of $j$,
we have $x_{j-1}=y_{j-1}$. Since $Y\in\mathcal{C}_e^+(A)$,
\cref{positivecircuit} gives
$a_{j-1,x_{j-1}}=a_{j-1,y_{j-1}}=-a_{j-1,y_j}$.

Applying \cref{intercambio}\ref{it:a} to $X$ and
$x_{j-1}<y_j<x_j$, we obtain $Z\in\mathcal{C}_e^+(A)$ with
$\underline{Z}=\underline{X}\cup y_j-x$ for some
$x\in\{x_j,\ldots,x_{r+1}\}$. Hence
$X\sim_{G_e(A)}Z$ by \cref{obs:uniformcircuit}. Moreover, the first
$j$ elements of $Z$ and $Y$ coincide.

If $Z=Y$, the result follows. Otherwise let $j'>j$ be the first
index at which $Z$ and $Y$ differ. By induction,
$d_{G_e(A)}(Z,Y)\le r+2-j'\le r+1-j$. Therefore
$d_{G_e(A)}(X,Y)\le1+d_{G_e(A)}(Z,Y)\le r+2-j$.

Assertion \ref{it:bb} follows by the analogous argument with the reverse order.
\end{proof}


\begin{example}\label{example1}
Using the matrix $A$ with $r=7$ rows of \cref{thmLOM}, we aim to find  a path from $\underline{X}=\{1,2,3,e=5,\textcolor{red}{8},10,12,15\}$ to
$\underline{Y}=\{1,2,3,e=5,\textcolor{red}{6},11,13,14\}$ in $G_e(A)$ as in the proof of \cref{e_primero_ultimo}~\ref{it:aa}. In this example $j=5$ (marked in red) and  $|\underline{X}\Delta \underline{Y}|=2(r+2-j)=8$. As $6<8$, we find   $\underline{Z_1}=\{1,2,3,5,6,8,10,15\}$  where  $\underline{X}\sim_{G_e(A)}\underline{Z_1}$ and $|\underline{Z_1}\Delta \underline{Y}|=6$. Proceeding in the same way as we did for $X$ and $Y$, but this time for $Z_1$ and $Y$, we find   $\underline{Z_2}=\{1,2,3,5,6,8,13,14\}$ where  $\underline{Y}\sim_{G_e(A)}\underline{Z_2}$ and $|\underline{Z_2}\Delta \underline{Z_1}|=4$. Finally, for $Z_1$ and $Z_2$  we find   $\underline{Z_3}=\{1,2,3,5,6,8,10,14\}$. Hence,  $X,Z_1,Z_3,Z_2,Y$ is a path of length $4=r+2-j$ from $X$ to $Y$ in $G_e(A)$.
\end{example}

Given a matrix $A\in\{+1,-1\}^{r\times n}$ and  $e\in [n]$, denote by $A^{\le}_{i}$ the submatrix of $A$ containing the set of rows $\{1,\ldots,i\}$ and the set of columns $\{1,\ldots,e\}$. Similarly,  denote by $A^{\ge}_{i}$  the submatrix of $A$ containing the set of rows $\{i,\ldots,r\}$ and the set of columns $\{e,\ldots,n\}$. See \cref{thmLOM} for an illustration. 
 Next, we prove that $d_{G_e(A)}(X,Y)\le r$ for $X,Y\in \mathcal{C}_e^+(A)$ both containing $e$ in the same position. 

\begin{lemma}\label{emismaposicion}
Let  $\M_A$ be a LOM of rank $r$   and let $X,Y\in \mathcal{C}_e^+(A)$ such that $e=x_{i}=y_{i}$, $1\le i\le r+1$. Then $d_{G_e(A)}(X,Y)\le r.$
\end{lemma}
\begin{proof}
If $X=Y$, there is nothing to prove. The case $i=1$ follows from
\cref{e_primero_ultimo}\ref{it:aa}, and the case $i=r+1$ follows from
\cref{e_primero_ultimo}\ref{it:bb}. We may therefore assume that
$2\le i\le r$.

Let
$\underline{X_1}=\{x_1,\ldots,x_{i-1},e\}$ and
$\underline{Y_1}=\{y_1,\ldots,y_{i-1},e\}$. By
\cref{positivecircuit},
$X_1,Y_1\in\mathcal{C}_e^+(A^{\le}_{i-1})$. The proof of
\cref{e_primero_ultimo}\ref{it:bb} gives a path from $X_1$ to $Y_1$ of
length at most $i-1$ in which consecutive circuits differ in one
element; if $X_1=Y_1$, take the path of length zero.

Similarly, let
$\underline{X_2}=\{e,x_{i+1},\ldots,x_{r+1}\}$ and
$\underline{Y_2}=\{e,y_{i+1},\ldots,y_{r+1}\}$. By
\cref{positivecircuit},
$X_2,Y_2\in\mathcal{C}_e^+(A^{\ge}_{i})$, and the proof of
\cref{e_primero_ultimo}\ref{it:aa} gives a path from $X_2$ to $Y_2$ of
length at most $r+1-i$ in which consecutive circuits differ in one
element.

Lift the first path to $\mathcal{C}_e^+(A)$ by adjoining
$\underline{X_2}$ to every support, and lift the second path by
adjoining $\underline{Y_1}$ to every support. By
\cref{positivecircuit}, all resulting signed sets are positive
circuits of $\M_A$. Consecutive supports have symmetric difference
$2$, so they are adjacent by \cref{obs:uniformcircuit}. Concatenating
the two lifted paths gives
\[
d_{G_e(A)}(X,Y)\le(i-1)+(r+1-i)=r,
\]
concluding the proof.
\end{proof}


We will need the following lemma in order to prove  \cref{LOMsdiamr}.

\begin{lemma}\label{cambioeposition}
Let   $\M_A$ be a  LOM of rank $r$ and  let $X\in \mathcal{C}_e^+(A)$. Then the following  hold.
\begin{enumerate}[label=(\alph*)]
  \item  If  $e=x_1$ then for every $i=2,\ldots,r$ there exists  $X_i\in \mathcal{C}_e^+(A^{\ge}_{i})$ \label{it:aaa}
  such that $e\in 
\underline{X_r}\subset \underline{X_{r-1}}\subset\cdots \subset \underline{X_2}\subset \underline{X}$.
\item If  $e=x_{r+1}$ then for every $i=1,\ldots,r-1$ there exists $X_i\in\mathcal{C}_e^+(A^{\le}_{i})$ \label{it:bbb} 
such that $e\in \underline{X_1}\subset \underline{X_{2}}\subset\cdots \subset \underline{X_{r-1}}\subset \underline{X}$.
\end{enumerate}
 \end{lemma}
 \begin{proof}
\ref{it:aaa} We construct $X_i$ by induction on $i$, for $2\le i\le r$.
For $i=2$, denote
$\underline{Z_1}=\{x_2,\ldots,x_{r+1}\}$ and note that
$Z_1\in\mathcal{C}^+(A^{\ge}_2)$ by \cref{positivecircuit}.
Applying \cref{intercambio}\ref{it:a} in $A^{\ge}_2$ to $Z_1$ and
$e<x_2$, we obtain $X_2\in\mathcal{C}_e^+(A^{\ge}_2)$ with
$\underline{X_2}=\underline{Z_1}\cup e-z_1\subset\underline{X}$ for
some $z_1\in\underline{Z_1}$.

Now suppose that for some $2\le i\le r-1$ we have constructed
$X_i\in\mathcal{C}_e^+(A^{\ge}_i)$ such that
$e\in\underline{X_i}\subset\underline{X_{i-1}}\subset\cdots
\subset\underline{X_2}\subset\underline{X}$ and
$\underline{X_i}=\underline{Z_{i-1}}\cup e-z_{i-1}$, where
$Z_{i-1}\in\mathcal{C}^+(A^{\ge}_i)$ and
$z_{i-1}\in\underline{Z_{i-1}}$.

Set
$\underline{Z_i}=\underline{Z_{i-1}}-z_{i-1}
=\underline{X_i}-e$. By \cref{positivecircuit},
$Z_i\in\mathcal{C}^+(A^{\ge}_{i+1})$. Applying
\cref{intercambio}\ref{it:a} in $A^{\ge}_{i+1}$ to $Z_i$ and the element
$e$, which is smaller than every element of $\underline{Z_i}$, we
obtain $X_{i+1}\in\mathcal{C}_e^+(A^{\ge}_{i+1})$ with
$\underline{X_{i+1}}=\underline{Z_i}\cup e-z_i$ for some
$z_i\in\underline{Z_i}$. Hence
$e\in\underline{X_{i+1}}\subset\underline{X_i}$, completing the
induction.

Assertion \ref{it:bbb} follows analogously by reversing the order.
\end{proof}

\LOM*

\begin{proof}
We  show that $d_{G_e(A)}(C,C')\le r$ for any $C,C'\in \mathcal{C}^+_e(A)$. The result then follows from \Cref{thm:main}, since $\M_A$ is uniform.

Let $s,q\ge0$ with  $s+q=r$ and write $\underline{C}=\{x_1,\ldots, x_{s},e,y_1,\ldots,y_q\}$, where we use the convention that $\underline{C}=\{x_1,\ldots, x_{r},e\}$ if $s=r$ and $\underline{C}=\{e,y_1,\ldots,y_r\}$ if $q=r$. Similarly,  let $s',q'\ge0$ with $s'+q'=r$ and write $\underline{C'}=\{x'_1,\ldots, x'_{s'},e,y'_1,\ldots,y'_{q'}\}$.  
If $s=s'$ (and so $q=q'$) then the theorem holds by \Cref{emismaposicion}. By symmetry, we may suppose that $s>s'$ and $q'>q$.

First, we construct positive circuits $Z_i\in  \mathcal{C}_e^+(A)$ for every $i=s',\ldots,s$, which will form part of the path between   $C$ and $C'$ in the graph $G_e(A)$, as follows.
Let $X_s\in \mathcal{C}_e^+(A^{\le}_{s})$ be such that $\underline{X_s}=\{x_1,\ldots, x_{s},e\}$ and note  that $A^{\le}_{s}$ is a matrix with $s$ rows.
Thus, applying   \cref{cambioeposition}\ref{it:bbb} in matrix $A^{\le}_{s}$ to $X_s$ for every $i=1,\ldots,s-1$, there exists $X_i\in \mathcal{C}_e^+(A^{\le}_{i})$ 
such that $e\in 
\underline{X_{1}}\subset \underline{X_{2}}\subset\cdots\subset \underline{X_s}$. 
Similarly,  we set $\underline{Y'_{s'+1}}=\{e,y'_1,\ldots,y'_{q'}\}$ and note that $Y'_{s'+1}\in \mathcal{C}_e^+(A^{\ge}_{s'+1})$. 
Applying  \cref{cambioeposition}\ref{it:aaa} in matrix $A^{\ge}_{s'+1}$ to $Y'_{s'+1}$,  we obtain that for every $i=s'+2,\ldots,r$,  there exists $Y'_i\in \mathcal{C}_e^+(A^{\ge}_{i})$ 
such that 
$e\in 
\underline{Y'_{r}}\subset \underline{Y'_{r-1}}\subset\cdots \subset \underline{Y'_{s}}\subset \cdots \subset \underline{Y'_{s'+1}}$. 

For notational convenience, set
$\underline{X_0}=\underline{Y'_{r+1}}=\{e\}$. For every
$i=s',\ldots,s$, define
\[
\underline{Z_i}=\underline{X_i}\cup\underline{Y'_{i+1}}.
\]
By \cref{positivecircuit}, $Z_i\in\mathcal{C}_e^+(A)$: the sign
conditions before $e$ are inherited from $X_i$, while those after
$e$ are inherited from $Y'_{i+1}$. Furthermore, the inclusions in
the two chains give
$|\underline{Z_i}\Delta\underline{Z_{i+1}}|=2$ for every
$i=s',\ldots,s-1$. Hence
$Z_i\sim_{G_e(A)}Z_{i+1}$ by \cref{obs:uniformcircuit}, and therefore
\[
d_{G_e(A)}(Z_s,Z_{s'})\le s-s'.
\]


Secondly,
$\{x_1,\ldots,x_s,e\}\subseteq
\underline{C}\cap\underline{Z_s}$. If $C\neq Z_s$, then the first
index $j$ for which $C$ and $Z_s$ differ satisfies $j\ge s+2$, and
\cref{e_primero_ultimo}\ref{it:aa} gives
$d_{G_e(A)}(C,Z_s)\le r+2-j\le q$.
The same inequality is trivial if $C=Z_s$.

Analogously, if $C'\neq Z_{s'}$, then the last index $j$ for which
$C'$ and $Z_{s'}$ differ satisfies $j\le s'$, and
\cref{e_primero_ultimo}\ref{it:bb} gives
$d_{G_e(A)}(C',Z_{s'})\le s'$.
Again, the inequality is trivial if $C'=Z_{s'}$. Therefore,
\[
d_{G_e(A)}(C,C')
\le d_{G_e(A)}(C,Z_s)
   +d_{G_e(A)}(Z_s,Z_{s'})
   +d_{G_e(A)}(Z_{s'},C')
\le q+(s-s')+s'=r,
\]
concluding the proof.
\end{proof}
\begin{example}
Using  matrix $A$ of \cref{thmLOM}, we aim to find  a path from 
$$\underline{C}=\{x_1,x_2,x_3,e,y_1,\ldots,y_4\}=\{1,2,3,e=5,8,10,12,15\}$$ to
$$\underline{C'}=\{x'_1,e,y'_1,\ldots,y'_6\}=\{4,e=5,6,7,9,11,13,14\}$$ in $G_e(A)$ as in the proof of Theorem \ref{LOMsdiamr}. 
In this example, $s=3,q=4,s'=1$ and $q'=6$. On the one hand, we find $Z_i\in \mathcal{C}_e^+(A)$ for $i=1,2,3$, where 
$\underline{Z_i}=\underline{X_i}\cup \underline{Y'_{i+1}}$ and $$\underline{X_{1}}=\{3,e\}\subset \underline{X_{2}}=\{1,3,e\}\subset \underline{X_3}=\{1,2,3,e\},$$ $$\underline{Y'_{4}}=\{e,6,11,13,14\}\subset \underline{Y'_{3}}=\{e,6,9,11,13,14\}\subset \underline{Y'_2}=\{e,6,7,9,11,13,14\}.$$ 
On the other hand, note that $d_{G_e(A)}(C,Z_3)=4$, as shown in Example \ref{example1}, and that $C'\sim_{G_e(A)}Z_1$. Therefore,  the path of length $7$ in $G_e(A)$  from $C$ to $C'$  has internal vertices
\[\{1,2,3,e,6,8,10,15\}
\rightarrow \{1,2,3,e,6,8,10,14\}
\rightarrow \{1,2,3,e,6,8,13,14\}
\rightarrow \underline{Z_3}
\rightarrow \underline{Z_2}
\rightarrow \underline{Z_1}.
\]

\end{example}

\begin{figure}[htb]
\begin{center}
 \includegraphics[width=1\textwidth]{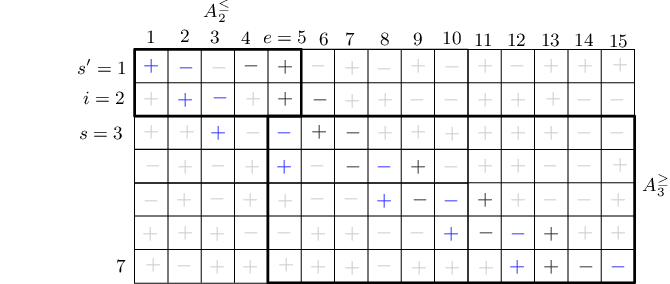}
 \caption{A matrix $A\in \{+1,-1\}^{7\times 15}$ and the submatrices $A_2^{\le}$ and $A_3^{\ge}$} \label{thmLOM}
 \end{center}
\end{figure}

\paragraph{Acknowledgments.} We thank Marco Caoduro for introducing us to the original problem on $0$-embracing sets, Francisco Santos and Frédéric Meunier for helpful comments concerning the Hirsch conjecture, and Benjamin Schröter and Lorenzo Vecchi for initial discussions on the matroidal extension. The research was supported by the Lendület Programme of the Hungarian Academy of Sciences (LP2021-1/2021); by the Ministry of Innovation and Technology of Hungary from the National Research, Development and Innovation Fund (ADVANCED 150556 and 153096); and by the Dynasnet European Research Council Synergy project (ERC-2018-SYG 810115).
KK was supported through grant 
PID2022-137283NB-C22 funded by MICIU/AEI/10.13039/501100011033 by ERDF/EU and through the Severo Ochoa and Mar\'ia de Maeztu Program for Centers and Units of Excellence in R\&D (CEX2020-001084-M) and ANR project MIMETIQUE: ANR-25-CE48-4089-01. L. P. Montejano was supported by the Spanish Ministry of Science and
Innovation due to the grants PID2023-146925OB-I00 and PID2023-148202OB-C21.

\paragraph{Declaration of generative AI and AI-assisted technologies in the manuscript preparation process.} During the preparation of this work, the authors made modest and conservative use of
ChatGPT (OpenAI, GPT-5.6 Sol) as an editorial and verification aid, primarily to improve
language and presentation, to identify possible inconsistencies, and to suggest local
revisions. All mathematical arguments and suggested changes were independently checked
by the authors. After using this tool, the authors reviewed and edited the content as needed
and take full responsibility for the content of the published article.

\bibliographystyle{abbrv}
\bibliography{lit}

\end{document}